\documentclass[11pt]{article}
\usepackage{amsthm,amssymb,amsmath}
\usepackage{epsfig}
\usepackage{verbatim}

\title{Strengthened Brooks Theorem for digraphs of girth three}

\author{Ararat Harutyunyan\thanks{Research supported by FQRNT
  (Le Fonds qu\'{e}b\'{e}cois de la recherche sur la nature et les technologies)
  doctoral scholarship.}\\
  {Department of Mathematics}\\
  {Simon Fraser University}\\
  {Burnaby, B.C. V5A 1S6} \\
  email: {\tt aha43@sfu.ca}
\and
  Bojan Mohar\thanks{Supported in part by an NSERC Discovery Grant (Canada),
  by the Canada Research Chair program, and by the
  Research Grant P1--0297 of ARRS (Slovenia).}~\thanks{On leave from:
  IMFM \& FMF, Department of Mathematics, University of Ljubljana, Ljubljana,
  Slovenia.}\\
  {Department of Mathematics}\\
  {Simon Fraser University}\\
  {Burnaby, B.C. V5A 1S6} \\
  email: {\tt mohar@sfu.ca}
}

\newtheorem{theorem}{Theorem}[section]
\newtheorem{lemma}[theorem]{Lemma}
\newtheorem{corollary}[theorem]{Corollary}
\newtheorem{proposition}[theorem]{Proposition}
\newtheorem{conjecture}[theorem]{Conjecture}

\newtheorem{question}[theorem]{Question}

\newcommand{\DEF}[1]{{\em #1\/}}

\newcommand{\C}{{\cal C}}
\newcommand{\PP}{\mathbb P}
\newcommand{\EE}{\mathbb E}
\newcommand{\tDelta}{\tilde{\Delta}}

\begin{document}

\maketitle

\begin{abstract}
Brooks' Theorem states that a connected graph $G$ of maximum
degree $\Delta$ has chromatic number at most $\Delta$, unless $G$
is an odd cycle or a complete graph. A result of Johansson
\cite{J1996} shows that if $G$ is triangle-free, then the
chromatic number drops to $O(\Delta / \log \Delta)$. In this
paper, we derive a weak analog for the chromatic number of
digraphs. We show that every (loopless) digraph $D$ without
directed cycles of length two has chromatic number $\chi(D) \leq
(1-e^{-13}) \tilde{\Delta}$, where $\tDelta$ is the maximum
geometric mean of the out-degree and in-degree of a vertex in $D$,
when $\tDelta$ is sufficiently large. As a corollary it is proved
that there exists an absolute constant $\alpha < 1$ such that
$\chi(D) \leq \alpha (\tDelta + 1)$ for every $\tDelta > 2$.
\end{abstract}


\section{Introduction}

Brooks' Theorem states that if $G$ is a connected graph with
maximum degree $\Delta$, then $\chi(G) \leq \Delta + 1$, where
equality is attained only for odd cycles and complete graphs. The
presence of triangles has significant influence on the chromatic
number of a graph. A result of Johansson \cite{J1996} states that
if $G$ is triangle-free, then $\chi(G) = O \left(\Delta / \log
\Delta \right)$. In this note, we study the chromatic number of
digraphs \cite{BFJKM2004}, \cite{M2003}, \cite{N1982} and show
that Brooks' Theorem for digraphs can also be improved when we
forbid directed cycles of length 2.

\subsection*{Digraph colorings and the Brooks Theorem}

Let $D$ be a (loopless) digraph. A vertex set $A \subset V(D)$ is
called \DEF{acyclic} if the induced subdigraph $D[A]$ has no
directed cycles. A \DEF{$k$-coloring} of $D$ is a partition of
$V(D)$ into $k$ acyclic sets. The minimum integer $k$ for which
there exists a $k$-coloring of $D$ is the \DEF{chromatic number}
$\chi(D)$ of the digraph D. The above definition of the chromatic
number of a digraph was first introduced by Neumann-Lara
\cite{N1982}. The same notion was independently introduced much
later by the second author when considering the circular chromatic
number of weighted (directed or undirected) graphs \cite{M2003}.
The chromatic number of digraphs was further investigated by Bokal
et al.~\cite{BFJKM2004}. The notion of chromatic number of a
digraph shares many properties with the notion of the chromatic
number of undirected graphs. Note that if $G$ is an undirected
graph, and $D$ is the digraph obtained from $G$ by replacing each
edge with the pair of oppositely directed arcs joining the same
pair of vertices, then $\chi(D) = \chi(G)$ since any two adjacent
vertices in $D$ induce a directed cycle of length two. Another
useful observation is that a $k$-coloring of a graph $G$ is a
$k$-coloring of a digraph $D$, where $D$ is a digraph obtained
from assigning arbitrary orientations to the edges of $G$. Mohar
\cite{M2010} provides some further evidence for the close
relationship between the chromatic number of a digraph and the
usual chromatic number. For digraphs, a version of Brooks' theorem
was proved in \cite{M2010}. Note that a digraph $D$ is
\DEF{$k$-critical} if $\chi(D) = k$, and $\chi(H) < k$ for every
proper subdigraph $H$ of $D$.

\begin{theorem}[\cite{M2010}]
\label{th:0}
Suppose that $D$ is a $k$-critical digraph in which for every
vertex $v \in V(D)$, $d^{+}(v) = d^{-}(v) = k-1$. Then one of the
following cases occurs:
\begin{enumerate}
\item $k=2$ and $D$ is a directed cycle of length $n \geq 2.$
\item $k=3$ and $D$ is a bidirected cycle of odd length $n \geq 3$.
\item $D$ is bidirected complete graph of order $k \geq 4$.
\end{enumerate}
\end{theorem}

A tight upper bound on the chromatic number of a digraph was first
given by Neumann-Lara \cite{N1982}.

\begin{theorem}[\cite{N1982}]
\label{th:01}
Let $D$ be a digraph and denote by $\Delta_o$ and $\Delta_i$ the
maximum out-degree and in-degree of $D$, respectively. Then
$$\chi(D) \leq \min \{\Delta_o,\Delta_i\} + 1.$$
\end{theorem}

In this note, we study improvements of this result using the
following substitute for the maximum degree. If $D$ is a digraph,
we let
$$
   \tDelta = \tDelta(D) =
   \max \{\sqrt{d^{+}(v)d^{-}(v)} \mid v \in V(D)\}
$$
be the maximum geometric mean of the in-degree and
out-degree of the vertices. Observe that $\tDelta \leq
\frac{1}{2}(\Delta_o + \Delta_i)$, by the arithmetic-geometric
mean inequality (where $\Delta_o$ and $\Delta_i$ are as in Theorem
\ref{th:01}). We show that when $\tDelta$ is large (roughly
$\tDelta \geq 10^{10}$), then every digraph $D$ without digons has
$\chi(D) \leq \alpha \tDelta$, for some absolute constant $\alpha
< 1$. We do not make an attempt to optimize $\alpha$, but show
that $\alpha = 1 - e^{-13}$ suffices. To improve the value of
$\alpha$ significantly, a new approach may be required.

It may be true that the following analog of Johansson's result
holds for digon-free digraphs, as conjectured by McDiarmid and
Mohar \cite{MM2002}.

\begin{conjecture} \label{conj:1}
Every digraph $D$ without digons has $\chi(D) =
O(\frac{\tDelta}{\log \tDelta})$.
\end{conjecture}

If true, this result would be asymptotically best possible in view
of the chromatic number of random tournaments of order $n$, whose
chromatic number is $\Omega(\frac{n}{\log n})$ and
$\tDelta > \left( \frac{1}{2} - o(1) \right)n$, as shown by
Erd\H{o}s et al.~\cite{EGK1991}.

We also believe that the following conjecture of Reed generalizes
to digraphs without digons.

\begin{conjecture}[\cite{R1998}] \label{conj:2}
Let $\Delta$ be the maximum degree of (an undirected) graph $G$,
and let $\omega$ be the size of the largest clique. Then
$$\chi(G) \leq \left \lceil \frac{\Delta + 1 + \omega}{2} \right \rceil.$$
\end{conjecture}

If we define $\omega = 1$ for digraphs without digons, we can pose
the following conjecture for digraphs.

\begin{conjecture} \label{conj:3}
Let $D$ be a $\Delta$-regular digraph without digons. Then
$$\chi(D) \leq \left \lceil \frac{\Delta}{2} \right \rceil + 1.$$
\end{conjecture}

Conjecture \ref{conj:3} is trivial for $\Delta = 1$, and follows
from Lemma \ref{L:4} for $\Delta = 2, 3$. We believe that the
conjecture is also true for non-regular digraphs with $\Delta$
replaced by $\tDelta$.

\subsection*{Basic definitions and notation}

We end this section by introducing some terminology that we will
be using throughout the paper. The notation is standard and we
refer the reader to \cite{BG2001} for an extensive treatment of
digraphs. All digraphs in this paper are \DEF{simple}, i.e. there
are no loops or multiple arcs in the same direction. We use $xy$
to denote the arc joining vertices $x$ and $y$, where $x$ is the
\DEF{initial vertex} and $y$ is the \DEF{terminal vertex} of the
arc $xy$. We denote by $A(D)$ the set of arcs of the digraph $D$.
For $v \in V(D)$ and $e \in A(D)$, we denote by $D-v$ and $D-e$
the subdigraph of $D$ obtained by deleting $v$ and the subdigraph
obtained by removing $e$, respectively. We let $d^{+}_D (v)$ and
$d^{-}_D(v)$ denote the \DEF{out-degree} (the number of arcs whose
initial vertex is $v$) and the \DEF{in-degree} (the number of arcs
whose terminal vertex is $v$) of $v$ in $D$, respectively. The
subscript $D$ may be omitted if it is clear from the context. A
vertex $v$ is said to be \DEF{Eulerian} if $d^{+}(v) = d^{-}(v)$.
The digraph $D$ is \DEF{Eulerian} if every vertex in $D$ is
Eulerian. A digraph $D$ is \DEF{$\Delta$-regular} if $d^{+}(v) =
d^{-}(v) = \Delta$ for all $v \in V(D)$. We say that $u$ is an
\DEF{out-neighbor} (\DEF{in-neighbor}) of $v$ if $vu$ ($uv$) is an
arc. We denote by $N^{+}(v)$ and $N^{-}(v)$ the set of
out-neighbors and in-neighbors of $v$, respectively. The
\DEF{neighborhood} of $v$, denoted by $N(v)$, is defined as $N(v)=
N^{+}(v) \cup N^{-}(v)$. Every undirected graph $G$ determines a
\DEF{bidirected} digraph $D(G)$ that is obtained from $G$ by
replacing each edge with two oppositely directed edges joining the
same pair of vertices. If $D$ is a digraph, we let $G(D)$ be the
\DEF{underlying undirected graph} obtained from $D$ by
``forgetting" all orientations. A digraph $D$ is said to be
\DEF{(weakly) connected} if $G(D)$ is connected. The \DEF{blocks}
of a digraph $D$ are the maximal subdigraphs $D'$ of $D$ whose
underlying undirected graph $G(D')$ is 2-connected. A \DEF{cycle}
in a digraph $D$ is a cycle in $G(D)$ that does not use parallel
edges. A \DEF{directed cycle} in $D$ is a subdigraph forming a
directed closed walk in $D$ whose vertices are all distinct. A
directed cycle consisting of exactly two vertices is called a
\DEF{digon}.

The rest of the paper is organized as follows. In Section 2, we
improve Brooks' bound for digraphs that have sufficiently large
degrees. In Section 3, we consider the problem for arbitrary
degrees.

\section{Strengthening Brooks' Theorem for large $\tDelta$}

The main result in this section is the following theorem.

\begin{theorem} \label{th:1}
There is an absolute constant $\Delta_1$ such that every
digon-free digraph $D$ with $\tDelta = \tDelta(D) \geq \Delta_1$
has $\chi(D) \leq \left(1- e^{-13} \right) \tDelta$.
\end{theorem}

The rest of this section is the proof of Theorem \ref{th:1}. The
proof is a modification of an argument found in Molloy and Reed
\cite{MR2002} for usual coloring of undirected graphs. We first
note the following simple lemma.

\begin{lemma} \label{L:1}
Let $D$ be a digraph with maximum out-degree $\Delta_o$, and
suppose we have a partial proper coloring of $D$ with at most
$\Delta_o + 1 - r$ colors. Suppose that for every vertex $v$ there
are at least $r$ colors that appear on vertices in $N^{+}(v)$ at
least twice. Then $D$ is $\Delta_o + 1 - r$-colorable.
\end{lemma}

\begin{proof}
The proof is easy -- since many colors are repeated on the
out-neighborhood of $v$, there are many colors that are not used
on $N^{+}(v)$. Thus, one can greedily ``extend" the partial
coloring.
\end{proof}

\begin{proof}[Proof of Theorem \ref{th:1}]
We may assume that $c_1 \tDelta < d^{+}(v) < c_2\tDelta$
and $c_1 \tDelta < d^{-}(v) < c_2\tDelta$
for each $v \in V(D)$, where $c_1 = 1 - \frac{1}{3}e^{-11}$
and $c_2 = 1 + \frac{1}{3}e^{-11}$. If not, we
remove all the vertices $v$ not satisfying the above inequality
and obtain a coloring for the remaining graph with $\left
(1-e^{-13} \right) \tDelta$ colors. Now, if a vertex does not
satisfy the above condition either one of $d^{+}(v)$ or $d^{-}(v)$
is at most $c_1 \tDelta$ or one of $d^{+}(v)$ or $d^{-}(v)$ is at
most $\frac{1}{c_2} \tDelta$. Note that $1 - e^{-13}
> \max \{c_1, 1/c_2\}$. This ensures that there is a color
that either does not appear in the in-neighborhood or does not appear
in the out-neighborhood of $v$, allowing us to complete the coloring.

The core of the proof is probabilistic. We color the vertices
of $D$ randomly with $C$ colors, $C = \lfloor \tDelta/2 \rfloor$.
That is, for each vertex $v$ we assign $v$ a color from
$\{1,2,..., C\}$ uniformly at random. After the random coloring,
we uncolor all the vertices that are in a monochromatic directed
path of length at least 2. Clearly, this results in a proper
partial coloring of $D$ since $D$ has no digons. For each vertex
$v$, we are interested in the number of colors which are assigned
to at least two out-neighbors of $v$ and are retained by at least
two of these vertices. For analysis, it is better to define a
slightly simpler random variable. Let $v \in V(D)$. For each color
$i$, $1 \leq i \leq C$, let $O_i$ be the set of out-neighbors of
$v$ that have color $i$ assigned to them in the first phase. Let
$X_v$ be the number of colors $i$ for which $|O_i| \geq 2$ and
such that all vertices in $O_i$ retain their color after the
uncoloring process.

For every vertex $v$, we let $A_v$ be the event that $X_v$ is less
than $\frac{1}{2}e^{-11}\tDelta + 1$. We will show that with
positive probability none of the events $A_v$ occur. Then Lemma
\ref{L:1} will imply that $\chi(D) \leq (c_2-\frac{1}{2}e^{-11})
\tDelta \leq  (1 - e^{-13})\tDelta$, finishing the proof. We will
use the symmetric version of the Lov\'{a}sz Local Lemma (see for
example \cite{AS1992}). Note that the color assigned initially to
a vertex $u$ can affect $X_v$ only if $u$ and $v$ are joined by a
path of length at most 3. Thus, $A_v$ is mutually independent of
all except at most $ (2c_2\tDelta) + (2c_2\tDelta)^{2} +
(2c_2\tDelta)^3 + (2c_2\tDelta)^4 + (2c_2\tDelta)^5 +
(2c_2\tDelta)^6 \leq 100 \tDelta^{6}$ other events $A_w$.
Therefore, by the symmetric version of the Local Lemma, it
suffices to show that for each event $A_v$, $4 \cdot 100 \tDelta^6
\PP[A_v] < 1$. We will show that $\PP[A_v] < \tDelta^{-7}$. We do
this by proving the following two lemmas.

\begin{lemma} \label{L:2}
$\EE[X_v] \geq e^{-11}\tDelta - 1$.
\end{lemma}

\begin{proof}

Let $X'_v$ be the random variable denoting the number of colors
that are assigned to exactly two out-neighbors of $v$ and are
retained by both of these vertices. Clearly, $X_v \geq X'_v$ and
therefore it suffices to consider $\EE[X'_v]$.

Note that color $i$ will be counted by $X'_v$ if two vertices $u,w
\in N^{+}(v)$ are colored $i$ and no other vertex in $S = N(u)
\cup N^{+}(v) \cup N(w)$ is assigned color $i$. This will give us
a lower bound on $\EE[X'_v]$. There are $C$ choices for color $i$
and at least $\binom{c_1\tDelta}{2}$ choices for the set
$\{u,w\}$. The probability that no vertex in $S$ gets color $i$ is
at least $(1- \frac{1}{C})^{|S|} \geq  (1- \frac{1}{C})^{5c_2
\tDelta}$. Therefore, by linearity of expectation, we can
estimate:
\begin{eqnarray*}
\EE[X'_v] &\geq& C \binom{c_1\tDelta}{2} \left( \frac{1}{C}
\right)^2
\left(1- \frac{1}{C}\right)^{5c_2 \tDelta}\\
&\geq& c_1(c_1\tDelta - 1) \exp(-5c_2\tDelta / C - 1/C) \\
&\geq& \frac{\tDelta}{e^{11}}- 1
\end{eqnarray*}
for $\tDelta$ sufficiently large.
\end{proof}

\begin{lemma} \label{L:3}
$\PP \left[ | X_v - \EE[X_v] | > \log \tDelta \sqrt{\EE[X_v]} \,
\right] < \tDelta^{-7}$.
\end{lemma}

\begin{proof}
Let $AT_v$ be the random variable counting the number of colors
assigned to at least two out-neighbors of $v$, and $Del_v$ the
random variable that counts the number of colors assigned to at
least two out-neighbors of $v$ but removed from at least one of
them. Clearly, $X_v = AT_v - Del_v$ and therefore it suffices to
show that each of $AT_v$ and $Del_v$ are sufficiently concentrated
around their means. We will show that for $t = \frac{1}{2} \log
\tDelta \sqrt{\EE[X_v]}$ the following estimates hold:

\medskip

Claim 1: $\PP \left[|AT_v - \EE[AT_v]| > t \right] < 2
e^{-t^2/(8 \tDelta)}$.

\medskip

Claim 2: $\PP \left[|Del_v - \EE[Del_v]| > t \right]
  < 4 e^{-t^2/(100 \tDelta)}$.

\medskip
\noindent
The two above inequalities yield that, for $\tDelta$ sufficiently
large,
\begin{eqnarray*}
\PP[ | X_v - \EE[X_v] | > \log \tDelta \sqrt{\EE[X_v]}] &\leq& 2
e^{-\frac{t^2}{8 \tDelta}} + 4 e^{-\frac{t^2}{100 \tDelta}}\\
&\leq& \tDelta^{- \log \tDelta}\\
&<& \tDelta^{-7},
\end{eqnarray*}
as we require. So, it remains to establish both claims.

To prove Claim 1, we use a version of Azuma's inequality found in
\cite{MR2002}, called the Simple Concentration Bound.

\begin{theorem}[Simple Concentration Bound] \label{th:1.5}
Let $X$ be a random variable determined by $n$ independent trials
$T_1,..., T_n$, and satisfying the property that changing the
outcome of any single trial can affect $X$ by at most $c$. Then
$$\PP[|X-\EE[X]| > t] \leq 2e^{-\frac{t^2}{2c^2n}}. $$
\end{theorem}

Note that $AT_v$ depends only on the colors assigned to the
out-neighbors of $v$. Note that each random choice can affect $AT_v$
by at most 1. Therefore, we can take $c=1$ in
the Simple Concentration Bound for $X=AT_v$.
Since the choice of random color assignments are made
independently over the vertices and since $d^{+}(v) \leq c_2
\tDelta$, we immediately have the first claim.

For Claim 2, we use the following variant of Talagrand's
Inequality (see \cite{MR2002}).

\begin{theorem}[Talagrand's Inequality] \label{th:2}
Let $X$ be a nonnegative random variable, not equal to 0, which
is determined by $n$ independent trials, $T_1,\dots,T_n$ and
satisfyies the following conditions for some $c,r > 0$:
\begin{enumerate}
\item Changing the outcome of any single trial can affect $X$ by
at most $c$. \item For any $s$, if $X \geq s$, there are at most
$rs$ trials whose exposure certifies that $X \geq s$.
\end{enumerate}
Then for any $0 \leq \lambda \leq \EE[X]$,
$$ \PP \left[|X-\EE[X]| >  \lambda + 60c \sqrt{r\EE[X]} \, \right]
   \leq 4e^{-\frac{\lambda^2}{8c^2r\EE[X]}}.
$$
\end{theorem}

We apply Talagrand's inequality to the random variable $Del_v$.
Note that we can take $c=1$ since any single random
color assignment can affect $Del_v$ by at most 1. Now, suppose
that $Del_v \geq s$. One can certify that $Del_v \geq s$ by
exposing, for each of the $s$ colors $i$, two random color
assignments in $N^{+}(v)$ that certify that at least two vertices
got color $i$, and exposing at most two other color assignments
which show that at least one vertex colored $i$ lost its color.
Therefore, $Del_v \geq s$ can be certified by exposing $4s$ random
choices, and hence we may take $r=4$ in Talagrand's inequality.
Note that $t= \frac{1}{2} \log \tDelta \sqrt{\EE[X_v]}
>\!\!> 60c \sqrt{r\EE[Del_v]} $ since $\EE[X_v] \geq \tDelta/e^{11} - 1$ and
$\EE[Del_v] \leq c_2 \tDelta$. Now, taking $\lambda$ in
Talagrand's inequality to be $\lambda = \frac{1}{2}t$, we obtain
that $\PP[|Del_v -\EE[Del_v]| > t] \leq \PP[|Del_v-\EE[Del_v]| >
\lambda + 60c \sqrt{r\EE[X]}]$. Therefore, provided that $ \lambda
\leq \EE[Del_v]$, we have the confirmed Claim 2.

It is sufficient to show that $\EE[Del_v] = \Omega (\tDelta)$,
since $\lambda = O(\log \tDelta \sqrt{\tDelta})$. The probability
that \emph{exactly} two vertices in $N^{+}(v)$ are assigned a
particular color $c$ is at least $\frac{c_1\tDelta^2}{2} C^{-2}
(1-1/C)^{c_2\tDelta} \approx 2e^{-10}$, a constant. It remains to
show that the probability that at least one of these vertices
loses its color is also (at least) a constant. We use Janson's
Inequality (see \cite{AS1992}). Let $u$ be one of the two vertices
colored $c$. We only compute the probability that $u$ gets
uncolored. We may assume that the other vertex colored $c$ is not
a neighbor of $u$ since this will only increase the probability.
We show that with large probability there exists a monochromatic
directed path of length at least 2 starting at $u$. Let $ \Omega =
N^{+}(u) \cup N^{++}(u)$, where $N^{++}(u)$ is the second
out-neighborhood of $u$. Each vertex in $\Omega$ is colored $c$
with probability $\frac{2}{\tDelta}$. Enumerate all the directed
paths of length 2 starting at $u$ and let $P_i$ be the $i^{th}$
path. Clearly, there are at least $(c_1 \tDelta)^2$ such paths
$P_i$. Let $A_i$ be the set of vertices of $P_i$, and denote by
$B_i$ the event that all vertices in $A_i$ receive the same color.
Then, clearly $\PP[B_i] = \frac{1}{(\lfloor \tDelta/2 \rfloor)^2}
\geq \frac{4}{\tDelta^2}$. Then, $\mu = \sum \PP[B_i] \geq
\frac{4}{\tDelta^2} \cdot (c_1\tDelta)^2 = 4c_1^2$. Now, if
$\delta = \sum_{i,j : A_i \cap A_j \neq \emptyset}\PP[B_i \cap
B_j]$ in Janson's Inequality satisfies $\delta < \mu$, then
applying Janson's Inequality, with the sets $A_i$ and events
$B_i$, we obtain that the probability that none of the events
$B_i$ occur is at most $e^{-1}$, and hence the probability that
$u$ does not retain its color is at least $1-e^{-1}$, as required.
Now, assume that $\delta \geq \mu$. The following gives an upper
bound on $\delta$:
\begin{eqnarray*}
\delta &=& \sum_{i,j : A_i \cap A_j \neq \emptyset}\PP[B_i \cap B_j]
~=~ \sum_{i, j: A_i \cap A_j \neq \emptyset} \frac{1}{(\lfloor \tDelta/2 \rfloor)^3} \\
&\leq& (c_2\tDelta)^2 \cdot 2c_2\tDelta \cdot
\frac{8}{(\tDelta-2)^3} < 32,
\end{eqnarray*}
for $\tDelta \geq 100$. Now, we apply Extended Janson's Inequality
(again see \cite{AS1992}). This inequality now implies that the
probability that none of the events $B_i$ occur is at most
$e^{-c_1^2/4}$, a constant. Therefore, by linearity of expectation
$\EE[Del_v] = \Omega(\tDelta)$.
\end{proof}

Clearly, since $\EE[X_v] \leq c_2 \tDelta$, Lemmas \ref{L:2} and
\ref{L:3} imply that $\PP[A_v] < \tDelta^{-7}$. This completes the
proof of Theorem \ref{th:1}.
\end{proof}

\section{Brooks' Theorem for small $\tDelta$}

The bound in Theorem \ref{th:1} is only useful for large
$\tDelta$. Rough estimates suggest that $\tDelta$ needs to be at
least in the order of $10^{10}$. The above approach is unlikely to
improve this bound significantly with a more detailed analysis.
In this section, we improve Brooks' Theorem for all values of
$\tDelta$. We achieve this by using a result on list colorings
found in \cite{HM2010}. List coloring of digraphs is defined
analogously to list coloring of undirected graphs. A precise
definition is given below.

Let $\C$ be a finite set of colors. Given a digraph $D$, let $L:
v\mapsto L(v)\subseteq \C$ be a \DEF{list-assignment} for $D$,
which assigns to each vertex $v\in V(D)$ a set of colors. The set
$L(v)$ is called the \DEF{list} (or the set of \DEF{admissible
colors}) for $v$. We say $D$ is \DEF{$L$-colorable} if there is an
\DEF{$L$-coloring} of $D$, i.e., each vertex $v$ is assigned a
color from $L(v)$ such that every color class induces an acyclic
subdigraph in $D$. $D$ is said to be \DEF{$k$-choosable} if $D$ is
$L$-colorable for every list-assignment $L$ with $|L(v)| \geq k$
for each $v \in V(D)$. We denote by $\chi_l(D)$ the smallest
integer $k$ for which $D$ is $k$-choosable.

The result characterizes the structure of non $L$-colorable
digraphs whose list sizes are one less than under Brooks'
condition.

\begin{theorem}[\cite{HM2010}] \label{th:4}
Let $D$ be a connected digraph, and $L$ an assignment of colors to
the vertices of $D$ such that $|L(v)| \geq d^{+}(v)$ if $d^{+}(v)
= d^{-}(v)$ and $|L(v)| \geq \min \{d^{+}(v), d^{-}(v)\} + 1$
otherwise. Suppose that $D$ is not $L$-colorable. Then $D$ is
Eulerian, $|L(v)|=d^{+}(v)$ for each $v \in V(D)$, and every block
of $D$ is one of the following:
\begin{enumerate}
\item[\rm (a)] a directed cycle (possibly a digon),
\item[\rm (b)] an odd bidirected cycle, or
\item[\rm (c)] a bidirected complete digraph.
\end{enumerate}
\end{theorem}

Now, we can state the next result of this section.
\begin{lemma} \label{L:4}
Let $D$ be a connected digraph without digons, and let $\tDelta =
\tDelta(D)$. If $\tDelta > 1$, then $\chi_l(D) \leq \lceil \tDelta
\rceil$.
\end{lemma}

\begin{proof}

We apply Theorem \ref{th:4} with all lists $L(v)$, $v \in V(D)$
having cardinality $\lceil \tDelta \rceil$. It is clear that the
conditions of Theorem \ref{th:4} are satisfied for every Eulerian
vertex $v$. It is easy to verify that the conditions are also
satisfied for non-Eulerian vertices. Now, if $D$ is not
$L$-colorable, then by Theorem \ref{th:4}, $D$ is Eulerian and
$d^{+}(v) = \lceil \tDelta \rceil$ for every vertex $v$. This
implies that $D$ is $\lceil \tDelta \rceil$-regular. Now, the
conclusion of Theorem \ref{th:4} implies that $D$ consists of a
single block of type (a), (b) or (c). This means that either $D$
is a directed cycle (and hence $\tDelta = 1$), or $D$ contains a
digon, a contradiction. This completes the proof.
\end{proof}

We can now prove the main result of this section, which improves
Brooks' bound for all digraphs without digons.

\begin{theorem} \label{th:5}
Let $D$ be a connected digraph without digons, and let $\tDelta =
\tDelta(D)$. If $\tDelta > 1$, then $\chi(D) \leq \alpha (\tDelta
+ 1) $ for some absolute constant $\alpha < 1$.
\end{theorem}

\begin{proof}
We define $\alpha = \max \left \{ \frac{\Delta_1}{\Delta_1 + 1},
1-e^{-13} \right \}$, where $\Delta_1$ is the constant in the
statement of Theorem \ref{th:1}. Now, if $\tDelta < \Delta_1$ then
by Lemma \ref{L:4}, it follows that $\chi(D) \leq \lceil \tDelta
\rceil \leq \alpha (\tDelta + 1)$. If $\tDelta \geq \Delta_1$,
then by Theorem \ref{th:1} we obtain that $\chi(D) \leq \left(1-
e^{-13} \right) \tDelta \leq \alpha (\tDelta + 1)$, as required.
\end{proof}

An interesting question to consider is the tightness of the bound
of Lemma \ref{L:4}. It is easy to see that the bound is tight for
$\lceil \tDelta \rceil = 2$ by considering, for example, a
directed cycle with an additional chord or a digraph consisting of
two directed triangles sharing a common vertex. The graph in
Figure 1 shows that the bound is also tight for $\lceil \tDelta
\rceil = 3$. It is easy to verify that, up to symmetry, the
coloring outlined in the figure is the unique 2-coloring. Now,
adding an additional vertex, whose three out-neighbors are the
vertices of the middle triangle and the three in-neighbors are the
remaining vertices, we obtain a 3-regular digraph where three
colors are required to complete the coloring.

Another example of a digon-free 3-regular digraph on 7 vertices
requiring three colors is the following. Take the Fano Plane and
label its points by 1,2,...,7. For every line of the Fano plane
containing points $a, b, c$, take a directed cycle through $a,b,c$
(with either orientation). There is a unique directed 3-cycle
through any two vertices because every two points line in exactly
one line. This shows that the Fano plane digraphs are not isomorphic
to the digraph from the previous paragraph.
Finally, it is easy to verify that the resulting digraph needs three
colors for coloring.

\begin{figure}[htb]
   \centering
   \includegraphics[height=5cm]{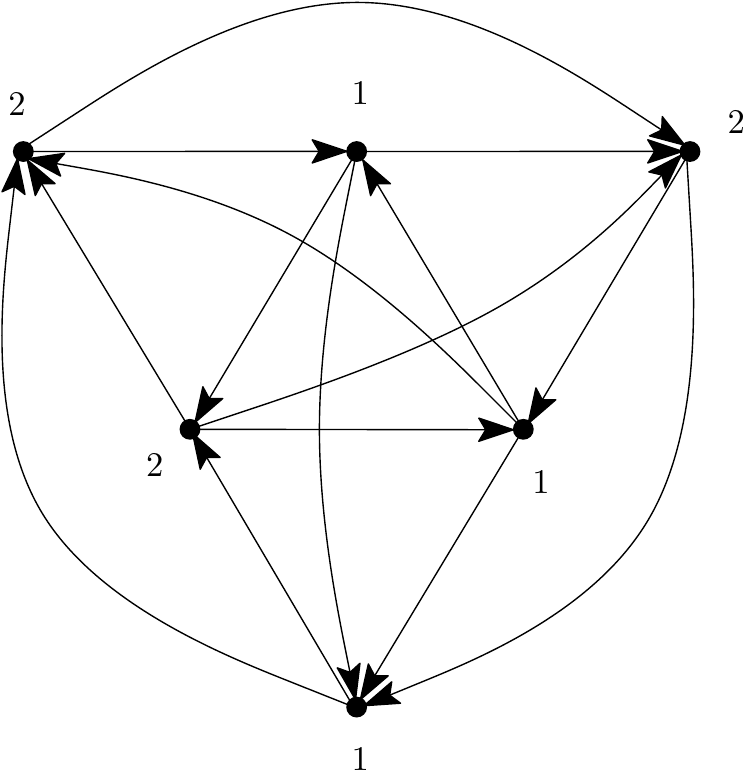}
   \caption{Constructing a $3$-regular digraph $D$ with $\chi(D) = 3$.}
   \label{fig:2}
\end{figure}
Note that the digraphs in the above examples are 3-regular
tournaments on 7 vertices. It is not hard to check that
every tournament on 9 vertices has $\lceil \tDelta \rceil =4$, and
yet is $3$-colorable. In general, we pose the following problem.

\begin{question} \label{Q:1}
What is the smallest integer $\Delta_0$ such that every digraph
$D$ without digons with $\lceil \tDelta(D) \rceil = \Delta_0$
satisfies $\chi(D) \leq \Delta_0 - 1$?
\end{question}

\bigskip
Note that this is a weak version of Conjecture \ref{conj:3}. By
Theorem \ref{th:1}, $\Delta_0$ exists. However, we believe that
$\Delta_0$ is small, possibly equal to 4. The following
proposition shows that the above holds for every $\lceil \tDelta
\rceil \geq \Delta_0$.

\begin{proposition} \label{prop:1}
Let $\Delta_0$ be defined as in Question \ref{Q:1}. Then every
digon-free digraph $D$ with $\lceil \tDelta(D) \rceil \geq
\Delta_0$ satisfies $\chi(D) \leq \lceil \tDelta(D) \rceil -1$.
\end{proposition}

\begin{proof}
The proof is by induction on $\lceil \tDelta \rceil$. If $\lceil
\tDelta \rceil = \Delta_0$ this holds by the definition of
$\Delta_0$. Otherwise, let $U$ be a maximal acyclic subset of $D$.
Then $\lceil \tDelta(D-U) \rceil \leq \lceil \tDelta(D) \rceil - 1$
for otherwise $U$ is not maximal. Since we can color $U$ by a
single color, we can apply the induction hypothesis to complete
the proof.
\end{proof}

As a corollary we get:

\begin{corollary}
There exists a positive constant $\alpha < 1$ such that for every
digon-free digraph $D$ with $\lceil \tDelta(D) \rceil \geq
\Delta_0$, $\chi(D) \leq \alpha \lceil \tDelta \rceil$.
\end{corollary}

\begin{proof}
Let $\alpha = \max \left \{ \frac{\lceil \Delta_1 \rceil}{\lceil
\Delta_1 \rceil + 1}, 1-e^{-13} \right \}$, where $\Delta_1$ is
the constant in the statement of Theorem \ref{th:1}. Now, applying
Theorem \ref{th:1} or Proposition \ref{prop:1} gives the result.
\end{proof}


\begin{thebibliography}{99}

\bibitem{AS1992} N. Alon, J. Spencer, The Probabilistic Method,
Wiley, 1992.

\bibitem{BG2001}
J. Bang-Jensen, G. Gutin, Digraphs. Theory, Algorithms and
Applications, Springer, 2001.

\bibitem{BFJKM2004}
D. Bokal, G. Fijav\v{z}, M. Juvan, P. M. Kayll, B. Mohar, The
circular chromatic number of a digraph, J. Graph Theory 46 (2004)
227--240.

\bibitem{EGK1991}
P. Erd\H{o}s, J. Gimbel, D. Kratsch, Some extremal results in
cochromatic and dichromatic theory, J. Graph Theory 15 (1991)
579--585.

\bibitem{HM2010}
A. Harutyunyan, B. Mohar, Gallai's theorem for digraphs, preprint.

\bibitem{J1996}
A. Johansson, Asymptotic choice number for triangle free graphs,
DIMACS Technical Report (1996) 91--95.

\bibitem{MM2002}
C. McDiarmid, B. Mohar, private communication, 2002.

\bibitem{M2003}
B. Mohar, Circular colorings of edge-weighted graphs, Journal of
Graph Theory 43 (2003) 107--116.

\bibitem{M2010}
B. Mohar, Eigenvalues and colorings of digraphs, Linear Algebra
and its Applications 432 (2010) 2273--2277.

\bibitem{MR2002} M. Molloy, B. Reed, Graph Colouring and the
Probabilistic Method, Springer, 2002.

\bibitem{N1982}
V. Neumann-Lara, The dichromatic number of a digraph, J.~Combin.\
Theory, Ser. B 33 (1982) 265--270.

\bibitem{R1998}
B. Reed, $\omega, \Delta,$ and $\chi$,
J. Graph Theory 27 (1998) 177--212.

\end{thebibliography}
\end{document}